\documentclass[11pt]{article}

\usepackage{fullpage}

\usepackage{graphicx,hyperref,cite,amsmath,amsthm}

\newtheorem{theorem}{Theorem}
\newtheorem{lemma}[theorem]{Lemma}
\newtheorem{corollary}[theorem]{Corollary}

\newcommand{\Fan}{\mathop{\mathrm{Fan}}}
\newcommand{\CF}{\mathop{\mathrm{ComponentFamily}}}
\newcommand{\Family}{\mathcal{F}}
\newcommand{\Sparse}{\mathop{\mathrm{Sparse}}}
\newcommand{\Apex}{\mathop{\mathrm{Apex}}}
\newcommand{\BigSet}{\left\{\frac{i}{i+1}~\Big|~ i\ge 0\right\}\cup
\left\{\frac{3i+2j}{2i+j+1}~\Big|~ i\ge 1, j\in\{0,1,2\}\right\}}
\newcommand{\sigmaf}{\sigma_{\Family}}
\newcommand{\rhof}{\rho_{\Family}}

\title{Densities of Minor-Closed Graph Families}
\author{David Eppstein\\
\small Computer Science Department\\[-0.8ex]
\small University of California, Irvine\\[-0.8ex]
\small Irvine, California, USA}

\begin{document}
\maketitle

\begin{abstract}
We define the limiting density of a minor-closed family of simple graphs $\Family$ to be the smallest number $k$ such that every $n$-vertex graph in $\Family$ has at most $kn(1+o(1))$ edges, and we investigate the set of numbers that can be limiting densities. This set of numbers is countable, well-ordered, and closed; its order type is at least $\omega^\omega$. It is the closure of the set of densities of density-minimal graphs, graphs for which no minor has a greater ratio of edges to vertices. By analyzing density-minimal graphs of low densities, we find all limiting densities up to the first two cluster points of the set of limiting densities, $1$ and $3/2$. For multigraphs, the only possible limiting densities are the integers and the superparticular ratios $i/(i+1)$.
\end{abstract}

\section{Introduction}

Planar simple graphs with $n$ vertices have at most $3n-6$ edges. Outerplanar graphs have at most $2n-4$ edges.  Friendship graphs have $3(n-1)/2$ edges. Forests have at most $n-1$ edges. Matchings have at most $n/2$ edges. Where do the coefficients $3$, $2$, $3/2$, $1$, and $1/2$ of the leading terms in these bounds come from?

Planar graphs, forests, outerplanar graphs, and matchings all form instances of \emph{minor-closed families of simple graphs}, families of the graphs with the property that any \emph{minor} of a graph $G$ in the family (a simple graph formed from $G$ by contracting edges and removing edges and vertices) remains in the family. The friendship graphs (graphs in the form of $(n-1)/2$ triangles sharing a common vertex) are not minor-closed, but are the maximal graphs in another minor-closed family, the graphs formed by adding a single vertex to a matching. For any minor-closed family $\Family$ of simple graphs there exists a number $k$ such that every $n$-vertex graph in $\Family$ has at most $kn(1+o(1))$ edges~\cite{Kos-Comb-84,Tho-MPCPS-84,Tho-JCTB-01}. We define the \emph{limiting density} of $\Family$ to be the smallest number $k$ with this property. In other words, it is the coefficient of the leading linear term in the \emph{extremal function} of $\Family$, the function that maps a number $n$ to the maximum number of edges in an $n$-vertex graph in $\Family$.
We may rephrase our question more formally, then, as: which numbers can be limiting densities of minor-closed families?

As we show,  the set of limiting densities is countable, well-ordered, and topologically closed (Theorem~\ref{thm:cwoc}). Additionally, the set of limiting densities of minor-closed graph families is the closure of the set of densities of a certain family of finite graphs, the \emph{density-minimal graphs} for which no minor has a greater ratio of edges to vertices (Theorem~\ref{thm:dd}). To prove this we use a separator theorem for minor-closed families~\cite{AloSeyTho-JAMS-90} to find density-minimal graphs that belong to a given minor-closed family, have a repetitive structure that allows them to be made arbitrarily large, and are close to maximally dense.

By analyzing the structure of density-minimal graphs, we can identify the smallest two cluster points of the set of limiting densities, the numbers $1$ and $3/2$, and all of the other possible limiting densities that are at most $3/2$. Specifically, the limiting densities below 1 are the \emph{superparticular ratios} $i/(i+1)$ for $i=0,1,2,\dots$, and the corresponding density-minimal graphs are the $(i+1)$-vertex trees. The limiting densities between 1 and $3/2$ are the rational numbers of the form $3i/(2i+1)$, $(3i+2)/(2i+2)$, and $(3i+4)/(2i+3)$ for $i=1,2,3,\dots$; the corresponding density-minimal graphs include the friendship graphs and small modifications of these graphs. Beyond $3/2$ the pattern is less clear, but each number $2-1/i$ is a cluster point in the set of limiting densities, and $2$ is a cluster point of cluster points. Similarly the number $3$ is a cluster point of cluster points of cluster points, etc. We summarize this structure in Theorem~\ref{thm:bigset}.

One may also apply the theory of graph minors to families of \emph{multigraphs} allowing multiple edges between the same pair of vertices as well as multiple self-loops connecting a single vertex to itself. In this case the theory of limiting densities and density-minimal graphs is simpler: the only possible limiting densities for minor-closed families of multigraphs are the integers and the superparticular ratios, and the only limit point of the set of limiting densities is the number~$1$ (Theorem~\ref{thm:multi}).

\section{Related work}

\subsection{Limiting densities of minor-closed graph families}

A number of researchers have investigated the limiting densities of minor-closed graphs. It is known that every minor-closed family has bounded limiting density~\cite{Mad-MA-67} and that this density is at most $O(h\sqrt{\log h})$ for graphs with an $h$-vertex forbidden minor~\cite{Kos-Comb-84,Tho-MPCPS-84}. This asymptotic growth rate is tight: $K_h$-free graphs have limiting density $\Theta(h\sqrt{\log h})$, and the constant factor hidden in the $\Theta$-notation above is known~\cite{Tho-JCTB-01}.

There have also been similar investigations into the dependence of the limiting density of $H$-minor-free graphs on the number of vertices in $H$ when $H$ is  not a clique~\cite{KosPri-DM-08,MyeTho-Comb-05}. However these works have a different focus than ours: they concern either the asymptotic growth rate of the limiting density as a function of the forbidden minor size or, in some cases, the limiting densities or extremal functions of specific minor-closed families~\cite{ChuReeSey-ms-08,KosPri-DM-10,SonTho-JCTB-06,ThoWol-JCTB-08} rather than, as here, the structure of the set of possible limiting densities.

\subsection{Growth rates of minor-closed graph families}

Bernardi, Noy, and Welsh~\cite{BerNoyWel-JCTB-10} investigate a different set of real numbers defined from minor-closed graph families, their growth rates. The growth rate of a family of graphs is a number $c$ such that the number of $n$-vertex graphs in the family, with vertices labeled by a permutation of the numbers from $1$ to $n$, grows asymptotically as $n! c^{n+o(n)}$~\cite{SchZit-JCtB-94}. Bernardi et al. investigate the topological structure of the set of growth rates, show that this set is closed under the doubling operation, and determine all growth rates that are at most 2.25159.

\subsection{Upper density of infinite graphs}

For arbitrary graphs, not belonging to a minor-closed family, the density is often defined differently, as the ratio
$$|E| / \binom{|V|}{2} = \frac{2|E|}{|V|(|V|-1)}$$
of the number of edges that are present in the graph to the number of positions where an edge could exist. (This definition is not useful for minor-closed families: for any nontrivial minor-closed family, the density defined in this way necessarily approaches zero in the limit of large $n$.)

This definition of density can be extended to infinite graphs as the \emph{upper density}: the upper density of an infinite graph $G$ is the supremum of numbers $\alpha$ with the property that $G$ contains arbitrarily large subgraphs with density $\alpha$. Although defined in a very different way to our results here, the set of possible upper densities is again limited to a well-ordered countable set, consisting of $0$, $1$, and the superparticular ratios $i/(i+1)$~\cite[Exercise 12, p. 189]{Die-05}.

\section{Density-minimal graphs}
\label{sec:density-minimal}

We define the \emph{density} of a simple graph $G$, with $m$ edges and $n$ vertices, to be the ratio $m/n$. We say that $G$ is \emph{density-minimal} if no proper minor of $G$ has equal or greater density. Equivalently, a connected graph $G$ is density-minimal if there is no way of contracting some of the edges of $G$, compressing multiple adjacencies to a single edge, and removing self-loops, that produces a smaller graph with greater density: edge removals other than the ones necessary to form a simple graph are not helpful in producing dense minors. The \emph{rank} of a connected graph $G$, again with $m$ edges and $n$ vertices, is the number $m+1-n$ of independent cycles in $G$; we say that $G$ is \emph{rank-minimal} if no proper minor of $G$ has the same rank. Every density-minimal graph is also rank-minimal, because a smaller graph with equal rank would have greater density.

\begin{figure}[t]
\centering\includegraphics{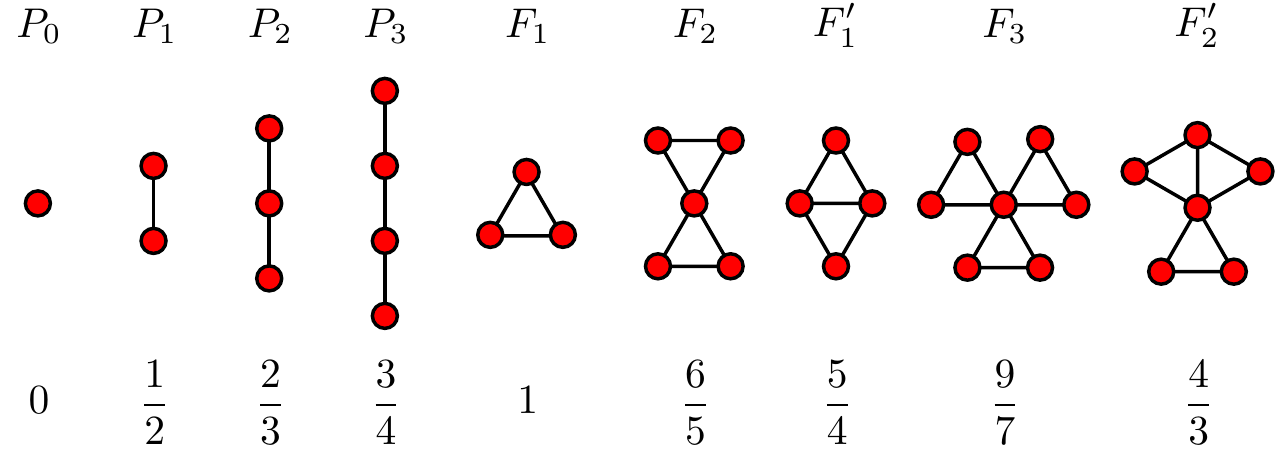}
\caption{Some density-minimal graphs and their densities.}
\label{fig:density-minimal}
\end{figure}

\subsection{Examples of density-minimal graphs}

Every tree is density-minimal: a tree with $m$ edges (such as the path $P_m$) has density $m/(m+1)$, and contracting edges in a tree can only produce a smaller tree with a smaller density. Additional examples of density-minimal graphs are provided by the \emph{friendship graphs} $F_i$ formed from a set of $i$ triangles by identifying one vertex from each triangle into a single supervertex;  these graphs are famous as the finite graphs in which every two vertices have exactly one common neighbor~\cite{ErdRenSos-SSMH-66}, but they also have the property that $F_i$ is density-minimal with density $3i/(2i+1)$. For, if we contract edges in $F_i$ to form a minor $H$, then $H$ must itself have the form of a friendship graph together possibly with some additional degree-one vertices connected to the central vertex. Removing the degree-one vertices can only increase the density of $H$, but once this is done $H$ must itself be a friendship graph with fewer triangles than $F_i$ and smaller density. This shows that every minor of $F_i$ has smaller density, so $F_i$ is density-minimal.

Let $F'_i$ be formed from the friendship graph $F_i$ by adding one more vertex, whose two neighbors are the two endpoints of any edge in $F_i$. Then a similar argument shows that $F'_i$ is density-minimal with density $(3i+2)/(2i+2)$. If a second vertex is added in the same way to produce a graph $F''_i$, then $F''_i$ is density-minimal with density $(3i+4)/(2i+3)$. However, adding a third vertex in the same way does not generally produce a density-minimal graph: its density is $3/2$, and (if $i>2$) one of the triangles of the friendship graph from which it was formed can be removed leaving a smaller minor with the same density.

These are not the only density-minimal graphs with these densities---a set of $i\ge 3$ triangles can be connected together at shared vertices to form cactus trees other than the friendship graphs with the same density---but as we now show, their densities are the only possible densities of density-minimal graphs in this numerical range.

\subsection{Bounding the rank of low-density density-minimal graphs}

In order to determine the possible densities of density-minimal graphs, it is helpful to have the following technical lemma, which allows us to restrict our attention to graphs of low rank.

\begin{lemma}
\label{lem:rank-limit-3/2}
Every biconnected graph of rank four or higher contains a minor of density at least $3/2$.
\end{lemma}

\begin{proof}
Let $G$ be biconnected with rank four or higher.
We perform an open ear decomposition of the graph (a partition of the edges into a sequence of subgraphs, the first of which is a cycle and the rest of which are simple paths, where the endpoints of each path belong to previous components of the decomposition)~\cite{Whi-TAMS-32}. The first four ears of this decomposition form a biconnected subgraph of $G$ with rank exactly four, so by replacing $G$ with this subgraph we may assume without loss of generality that the rank of $G$ is four. We may also assume without loss of generality that $G$ has no edges that could be contracted in a way that preserves both its rank and its biconnectivity, because otherwise we could replace $G$ with the graph formed by performing these contractions; in particular, this implies that every degree-two vertex of $G$ is part of a triangle. Additionally, no two degree-two vertices can be adjacent, for if they were then the third vertex of their triangle would be an articulation point, contradicting biconnectivity.

\begin{figure}[t]
\centering\includegraphics{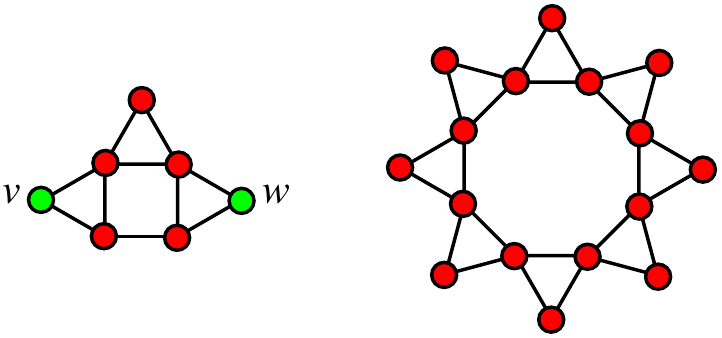}
\caption{Left: a rank-four graph with seven vertices. After removing the two degree-two vertices $v$ and $w$, the remaining graph $G''$ is a rank-two theta graph $\Theta(1,2,3)$; two of the edges on the length-three path of the theta are covered by the triangles containing $v$ and $w$, but the remaining edge at the bottom of the drawing can be contracted to produce a smaller biconnected rank-four graph. Right: a cycle of triangles, an example of a high-rank biconnected graph in which all minors have density at most $3/2$.}
\label{fig:rank-density}
\end{figure}

We now assert that $G$ has at most six vertices, and therefore has density at least $3/2$. For, suppose that $G$ had $n$ vertices with $n\ge 7$, and $n+3$ edges. In this case, since the number of edges is less than $3n/2$, there must be a vertex $v$ with degree two. Removing $v$ leaves a smaller graph $G'$ with $n-1\ge 6$ vertices and $n+1$ edges; again, the number of edges is less than three-halves of the number of vertices, so there must be a vertex $w$ that has degree two in $G'$ ($w$ cannot have degree one, because if it did then $G$ would have two adjacent degree-two vertices). If $w$ is not adjacent to $v$ in $G$, then the two neighbors of $w$ in $G'$ form a triangle as they do in $G$, for the same reason that the neighbors of $v$ form a triangle; if on the other hand $w$ is adjacent to $v$, then its two remaining neighbors in $G'$ must again form a triangle or else the edge $wx$ where $x$ is not adjacent to $v$ could be contracted preserving rank and biconnectivity.
Removing $w$ from $G'$ leaves a second smaller graph $G''$ with rank two. Additionally, since both $v$ and $w$ belonged to triangles of $G$, their removal cannot create any articulation points in $G''$, so $G''$ is biconnected.

But (by ear decomposition again) the only possible structure for a biconnected rank-two graph such as $G''$ is a theta graph $\Theta(a,b,c)$, in which two degree-three vertices are connected by three paths of lengths $a$, $b$, and $c$ respectively. If $G$ has seven or more vertices, then $G''$ has five or more vertices, and $a+b+c\ge 6$. If one of the paths of the theta graph has length three or more, then only two of the edges of this path can be part of triangles containing $v$ and $w$, and the third edge of the path can be contracted in $G$ to produce a smaller biconnected rank-four graph, contradicting our assumption that no such contraction exists; this case is shown in Figure~\ref{fig:rank-density}(left). In the remaining case, $G''=\Theta(2,2,2)=K_{2,3}$; only two of the six edges of $G''$ can be part of triangles involving $v$ and $w$, and any one of the four remaining edges can be contracted in $G$ to produce a smaller biconnected rank-four graph, again contradicting our assumption.

These contradictions show that the number $n$ of vertices in $G$ is at most six; since its rank is four, its density $(n+3)/n$ must be at least $3/2$.
\end{proof}

Increasing the rank past four without increasing the connectivity does not increase the density past the $3/2$ threshold of Lemma~\ref{lem:rank-limit-3/2}: there exist biconnected graphs of arbitrarily high rank in which the densest minors have density $3/2$, namely the cycles of triangles shown in Figure~\ref{fig:rank-density}(right).

\begin{figure}[t]
\centering\includegraphics{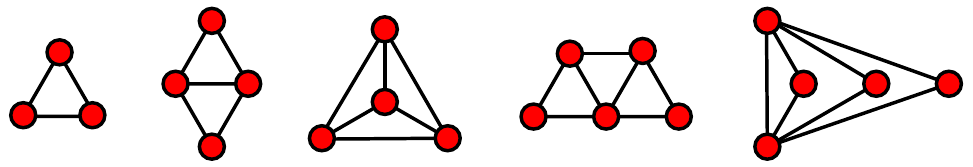}
\caption{The five rank-minimal biconnected graphs with rank between one and three.}
\label{fig:low-rank-blocks}
\end{figure}

A simple case analysis based on ear decompositions shows that there are exactly five rank-minimal biconnected graphs of rank between one and three: the triangle $F_1=K_3$, the diamond graph with four vertices and rank two, the complete graph $K_4$ with four vertices and rank three, and two different 2-trees with five vertices and rank three (Figure~\ref{fig:low-rank-blocks}).

\subsection{Classification of density-minimal graphs with low density}

\begin{lemma}
\label{lem:low-density}
Let a graph $G$ be density-minimal, with density $\Delta<3/2$. Then
$$\Delta\in\BigSet$$
and for each number $\Delta$ in this set there exists a density-minimal graph $G$ with density $\Delta$.
\end{lemma}

\begin{proof}
First, suppose that $\Delta<1$. Then $G$ must be acyclic and connected, for if it had a cycle it would have a triangle minor with density $1$ and not be density-minimal, and if it were disconnected then the densest of its components would be a minor with density at least as great as that of $G$ itself. But an acyclic connected graph is a tree, and has density $i/(i+1)$ where $i$ is its number of edges.

In the remaining cases, $1\le\Delta<3/2$. $G$ must be bridgeless, because any bridge could be contracted producing a denser graph.
By Lemma~\ref{lem:rank-limit-3/2}, each block (biconnected component) of $G$ must have rank at most three, for otherwise that block by itself would have a minor with density at least $3/2$, contradicting the assumption that $G$ is density-minimal. Additionally, each block must be rank-minimal, for otherwise $G$ itself would not be rank-minimal.  Therefore, each block must be a triangle, the diamond graph, or one of the two rank-three 2-trees. If there are two rank-three blocks, three rank-two blocks, or one rank-three block and one rank-two block, then those blocks alone (with the remaining blocks contracted away) would again form a minor with density at least $3/2$. The only remaining cases are a graph in which the blocks consist of $i$ triangles, with density $3i/(2i+1)$, a graph in which the blocks consist of $i$ triangles and one rank-two block, with density $(3i+5)/(2i+4)$, a graph in which the blocks consist of $i$ triangles and one rank-three block, with density $(3i+7)/(2i+5)$, or a graph in which the blocks consist of $i$ triangles and two rank-two blocks, with density $(3i+10)/(2i+7)$. Each of these densities belongs to the set specified in the lemma.

To show that each $\Delta$ in the set of densities stated in the lemma is the density of some density-minimal graph $G$, we need only recall the path graphs $P_i$, the friendship graphs $F_i$, and the graphs $F'_i$ ad $F''_i$ formed by adding one or two degree-two vertices to a friendship graph. As we have already argued, these graphs are density-minimal, and together they cover all the densities in the given set.
\end{proof}

The set of achievable densities allowed by Lemma~\ref{lem:low-density}, in numerical order up to the limit point $3/2$, is
$$0, \frac{1}{2}, \frac{2}{3}, \frac{3}{4}, \frac{4}{5}, \frac{5}{6}, \dots, 1,
\frac{6}{5}, \frac{5}{4}, \frac{9}{7}, \frac{4}{3},
\frac{15}{11}, \frac{11}{8}, \frac{18}{13}, \frac{7}{5}, \frac{24}{17}, \frac{17}{12}, \frac{27}{19}, \frac{10}{7}, \dots \frac{3}{2}.$$
Figure~\ref{fig:density-minimal} shows density-minimal graphs achieving some of these densities.

\section{Fans of graphs}

We now introduce a notation for constructing large graphs with a repetitive structure from a smaller model graph. Given a graph $G$, a proper subset $S$ of the vertices of $G$, and a positive integer $k$, we define the graph $\Fan(G,S,k)$ to be the union of $k$ copies of $G$, all sharing the same copies of the vertices in $S$ and having distinct copies of the vertices in $G\setminus S$. For instance:
\begin{itemize}
\item If $G$ is a triangle $uvw$, then $\Fan(G,\{u\},k)$ is the friendship graph $F_k$
\item For the same triangle $G=uvw$, $\Fan(G,\{u,v\},k)$ is a 2-tree formed by $k$ triangles sharing a common edge. Three of the graphs in Figure~\ref{fig:low-rank-blocks} take this form, for $k\in\{1,2,3\}$.
\item The complete bipartite graph $K_{a,b}$ is can be represented in multiple different ways as a fan: it is isomorphic to $\Fan(K_{d,b},S,a/d)$ where $d$ is any divisor of $a$ and $S$ is the $b$-vertex side of the bipartition of $K_{d,b}$, and symmetrically there is a representation as a fan for any divisor of $b$.
\item Two more examples are shown in Figure~\ref{fig:fans}.
\end{itemize}

\begin{figure}[t]
\centering\includegraphics{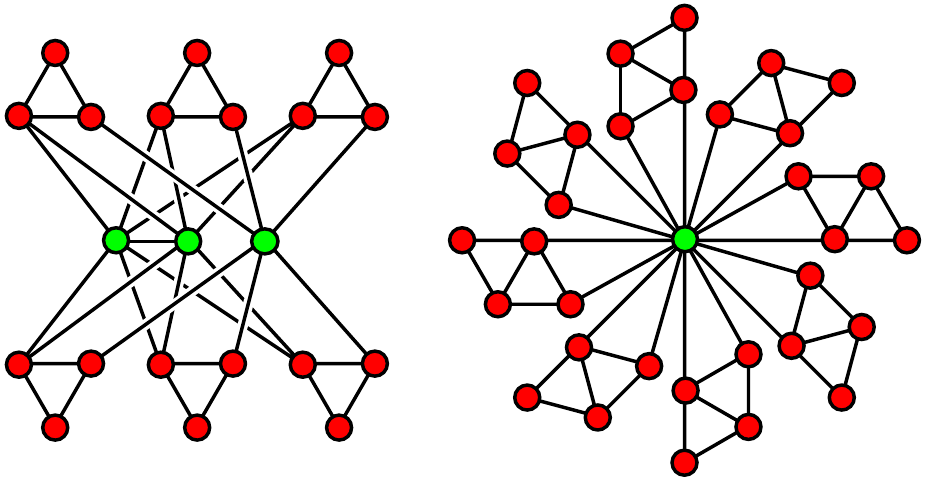}
\caption{Two fans of graphs.}
\label{fig:fans}
\end{figure}

\subsection{Basic observations about fans}

If $G$ has $n$ vertices and $S$ has $s$ vertices (where $s<n$ as we require $S$ to be a proper subset of the vertices), then $\Fan(G,S,k)$ has $k(n-s)+s=\Omega(k)$ vertices.  $\Fan(G,\emptyset,k)$ is the disconnected graph formed by $k$ disjoint copies of $G$;  however, if $G$ is connected and $S$ is nonempty then $\Fan(G,S,k)$ is also connected.

\begin{lemma}
\label{lem:fan-lb}
Every graph $\Fan(G,S,k)$ has at least $k$ vertices.
\end{lemma}

\begin{proof}
This follows immediately from the requirement that $S$ be a proper subset of the vertices of $G$; there is at least one vertex that does not belong to $S$, and that is replicated $k$ times in $\Fan(G,S,k)$.
\end{proof}

As the following lemma shows, it is not very restrictive to consider only those fans in which the central subset $S$ forms a clique. The advantage of restricting $S$ in this way is that, when considering minors of the fan, we do not have to consider the ways in which such a minor might add edges between vertices of $S$.

\begin{lemma}
\label{lem:fan-clique}
Let $G$ be a graph, let $S$ be a subset of the vertices of $G$ such that $G\setminus S$ is connected and every vertex in $S$ has a neighbor in $G\setminus S$, and let $G'$ be the graph formed from $G$ by adding edges between every pair of vertices in $S$. Then there exists a constant $c$ (depending on $G$ and $S$) with the property that, for every $k>c$, $\Fan(G',S,k-c)$ is a minor of $\Fan(G,S,k)$.
\end{lemma}

\begin{proof}
Let $K$ be a maximum clique in the subgraph induced in $G$ by $S$, and let $c=|S\setminus K|$. To form $\Fan(G',S,k-c)$ as a minor of $\Fan(G,S,k)$, simply contract one of the copies of $G$ onto each vertex in $S\setminus K$.
\end{proof}

For instance, in Figure~\ref{fig:fans}(left), the three central vertices do not form a clique, but they can be made into a clique by contracting one of the six copies of the outer subgraph into the rightmost of the three central vertices. Thus, in this example, we may take $c=1$.

\subsection{Densest minors of fans}

The following technical lemma will be used to compare the densities of minors of fans in which different copies of $G$ are contracted differently from each other.

\begin{lemma}
\label{lem:subst}
Let $a$, $b$, $c$, $d$, $e$, and $f$ be any six non-negative numbers, with $d+2e$, $d+2f$, and $d+e+f$ positive.
Then $(a+b+c)/(d+e+f)\le \max\{(a+2b)/(d+2e),(a+2c)/(d+2f)\}$, with equality only in the case that the two terms in the maximum are equal.
\end{lemma}

\begin{proof}
$(a+b+c)/(d+e+f)$ is a weighted average of the other two fractions with weights $(d+2e)/(2d+2e+2f)$ and $(d+2f)/(2d+2e+2f)$ respectively. As a weighted average with positive weights, it is not larger than the maximum of the two averaged values, and can be equal only when the two averaged values are equal.
\end{proof}

With the assumption that $S$ is a clique, any fan has a density-minimal minor that is also a fan with equal or greater density:

\begin{lemma}
\label{lem:fan-minimality}
Let $G$ be a connected graph, let $S$ be a proper subset of $G$ that induces a clique in $G$, let every vertex in $S$ be adjacent to at least one vertex in $G\setminus S$, and let $G\setminus S$ induce a connected subgraph of $G$. Then, for every $k$, there exists a densest minor of $\Fan(G,S,k)$ that is isomorphic to $\Fan(G',S',k)$, where $G'$ is some minor of $G$ (that may depend on $k)$ and $S'$ is the image of $S$ in $G'$. In addition if $S'$ is nonempty then $G'$ can be chosen so that $\Fan(G',S',k)$ is density-minimal; if $S'$ is empty then $G'$ can be chosen to be itself density-minimal.
\end{lemma}

\begin{proof}
Let $H$ be a densest minor of $\Fan(G,S,k)$, and let $S'$ be the image of $S$ in $H$. Note that it is not possible for all copies of $G$ to be contracted onto $S'$ in $H$, for (with $s=|S'|$) the density $(s-1)/2$ that would arise from this possibility is less than the density
$$\frac{s(s-1)/2+ks}{s+k}=\frac{s}{s+k}\cdot\frac{s-1}{2}+\frac{k}{s+k}s$$
of the graph in which, in each copy of $G$, the vertices that are not part of $S$ are contracted into a single vertex: as the formula shows, this latter graph's density is a weighted average of $(s-1)/2$ and $s$, and therefore exceeds $(s-1)/2$.

Suppose that the copies of $G$ in the fan are transformed in $H$ into two or more different minors, and let $G_1$ and $G_2$ be two of these minors. Define the integers $b$, $c$, and $a$ respectively to be the number of edges contributed to $H$ by minors of type $G_1$, the number of edges contributed to $H$ by minors of type $G_2$, and the number of remaining edges (including edges connecting vertices in $S'$. Similarly, define $e$, $f$, and $d$ respectively to be the number of vertices contributed by minors of type $G_1$, the number of vertices contributed by minors of type $G_2$, and the remaining number of vertices (including vertices in $S'$). Then the density of $H$ is $(a+b+c)/(d+e+f)$, the density of the minor of $\Fan(G,S,k)$ formed by replacing all copies of $G_2$ by $G_1$ is $(a+2b)/(d+2e)$, and the density of the minor formed by replacing all copies of $G_1$ by $G_2$ is $(a+2c)/(d+2f)$. By Lemma~\ref{lem:subst}, one of these two replacements provides a minor of $\Fan(G,S,k)$ that is at least as dense as $H$ but that has one fewer different type of minor of $G$. By induction on the number of types of minors of $G$ appearing in $H$, there is a densest minor of $\Fan(G,S,k)$ of the form $\Fan(G',S',k)$ where $G'$ is a minor of $G$. Among all choices of $G'$ leading to densest minors of this form, choose $G'$ to be minimal in the minor ordering.

If $S'=\emptyset$, $G'$ must be density-minimal, for any minor of $G'$ with equal or smaller density would have been chosen in place of $G'.$ If $S'\ne\emptyset$, $\Fan(G',S',k)$ must be density-minimal, for (again using Lemma~\ref{lem:subst} to reduce the number of different types of minor in the graph) substituting some copies of $G'$ for smaller minors could only produce a graph of equal density if these substitute minors could all be chosen to be isomorphic copies of a single minor $G''$, but then (because of the choice of $G'$ as giving the densest minor of this form) $\Fan(G'',S'',k)$ would have equal density to $\Fan(G',S',k)$, contradicting the choice of $G'$ as a minor-minimal graph whose fan has this density.
\end{proof}

\section{Limiting density from density-minimal graphs}

As we now show, the density-minimal graphs constructed in the previous two sections provide examples of minor-closed families, with the limiting density of the minor-closed family equal to the density of the density-minimal graph.

\subsection{Minor-closed families with the density of a given density-minimal graph}

For any graph $G$, define $\CF(G)$ to be the family of minors of graphs of the form $\Fan(G,\emptyset,k)$ for some positive $k$. The graphs in $\CF(G)$ are characterized by the property that every connected component is a minor of $G$. Clearly, $\CF(G)$ is minor-closed.

\begin{lemma}
\label{lem:cf-density}
If $G$ is density-minimal, then the limiting density of $\CF(G)$ equals the density of $G$.
\end{lemma}

\begin{proof}
$\CF(G)$ contains arbitrarily large graphs with this density, namely the graphs $\Fan(G,\emptyset,k)$. Therefore its limiting density is at least equal to the density of $G$. But the fact that its limiting density is no more than the density of $G$ is immediate, for if it contained any denser graph then that graph would have a dense component forming a minor of $G$ and contradicting the assumed density-minimality of~$G$.
\end{proof}

\begin{corollary}
\label{cor:cf}
The set of limiting densities of minor-closed graph families is a superset of the set of densities of density-minimal graphs.
\end{corollary}

\subsection{Gaps in the densities of density-minimal graphs}

The idea of designing a minor-closed family to achieve a given density comes up in a different way in the following lemma, which is the basis for our proof that the limiting densities form a well-ordered set.

\begin{lemma}
\label{lem:jump}
For every $\Delta\ge 0$ there exists a number $\delta>0$ such that the open interval $(\Delta,\Delta+\delta)$ does not contain the density of any density-minimal graph.
\end{lemma}

\begin{proof}
Let $\Sparse(\Delta)$ be the minor-closed family of graphs that do not contain any minor that is more dense than $\Delta$. By the Robertson--Seymour theorem, $\Sparse(\Delta)$ can be characterized by a finite set of minimal forbidden minors; a minimal forbidden minor for $\Sparse(\Delta)$ is a graph whose density is greater than $\Delta$ but all of whose proper minors have density less than or equal to $\Delta$. Thus, any minimal forbidden minor of $\Sparse(\Delta)$ must have density greater than $\Delta$; let $\Delta+\delta$ be the smallest density among these finitely many forbidden minors. Now suppose that $G$ is a density-minimal graph whose density is greater than $\Delta$. $G$ is not in $\Sparse(\Delta)$, so it has as a minor one of the forbidden minors of $\Sparse(\Delta)$, all of which have density at least $\Delta+\delta$. Because $G$ is density-minimal, its density is at least as great as that of its minor, and therefore is also at least $\Delta+\delta$. Therefore, $G$ cannot have density within the interval $(\Delta,\Delta+\delta)$.
\end{proof}

\section{Separators in minor-closed graph families}

We will use separators as a tool to find dense fans in any minor-closed family, so in this section we formalize the mathematics of separators in the form that we need.

If $G$ is any graph, define a \emph{separation} of $G$ to be a collection of subgraphs that partition the
edges of $G$. We call these subgraphs the \emph{separation components} of a separation. Define the \emph{separator} of a separation to be the set of vertices that belong to more than one separation component. The separator theorem of Alon, Seymour, and Thomas~\cite{AloSeyTho-JAMS-90} can be rephrased in terms of separations:

\begin{lemma}[Alon, Seymour, and Thomas~\cite{AloSeyTho-JAMS-90}]
\label{lem:sqrt-separator}
For any minor-closed graph family $\Family$ there exists a constant $\sigmaf$ such that
every $n$-vertex graph $G$ in $\Family$  has a separation of $G$ into two subgraphs $G_1$ and $G_2$, each having at least $n/3$ vertices, with at most $\sigmaf\sqrt n$ vertices in the separator.
\end{lemma}

For our purposes we need a form of separator theorem that produces much smaller separation components, with only constant size. This can be achieved by repeatedly applying the separator theorem of Lemma~\ref{lem:sqrt-separator} to larger components until they are small enough. A finer separation theorem of this type for planar graphs, proved in the same way from the planar separator theorem, has long been known~\cite[Theorem 3]{LipTar-SJC-80} and similar methods have been used as well in the context of minor-closed graphs (e.g. see~\cite[Lemma 3.4]{TazMul-DAM-09}). However, in order to control the density of the components, we need to be careful about how we measure the size of the separator, since a single separator vertex may appear in many components. We define the \emph{separator multiplicity} of a vertex $v$ in a separation of $G$ to be one less than the number of separation components containing $v$, so that a vertex has nonzero separator multiplicity if and only if it belongs to the separator, and we define the \emph{total multiplicity} of a separation to be the sum of the separator multiplicities of the vertices. If a vertex has separator multiplicity zero we say that it is an \emph{internal vertex} of its separation component.

\begin{lemma}
\label{lem:pulverize}
For any minor closed graph family $\Family$ there exists a constant $\rhof$ such that, for every $0<\epsilon<1$ and every $n$-vertex graph $G$, there is a separation of $G$ into subgraphs, each having at most $\rhof/\epsilon^2$ vertices, with total multiplicity at most $\epsilon n$.
\end{lemma}

\begin{proof}
Set
$$X=\frac{2\sigmaf}{\sqrt{2/3}+\sqrt{1/3}-1} \mbox{~and~} \rhof=3X^2.$$

We prove by induction on $n$ a stronger version of the lemma: whenever $n\ge (X/\epsilon)^2$, there exists a separation in which the size of each separation component is as specified and for which the total multiplicity is at most $\epsilon\, n - X\sqrt n.$
For $n<(X/\epsilon)^2$ the induction does not go through because $\epsilon\, n - X\sqrt n$ is negative, but in this case the lemma itself follows trivially since we may choose a separation with a single separation component (all of $G$) and no vertices in the separator; the total multiplicity is $0$, which is is at most $\epsilon n$ as desired.
As a base case for the induction, whenever $(X/\epsilon)^2\le n\le 3(X/\epsilon)^2=\rhof/\epsilon^2$, we again choose a separation with a single separation component and no vertices in the separator; for $n$ at least this large, $\epsilon\, n - X\sqrt n\ge 0$ which again exceeds the zero value of the total multiplicity.

For larger values of $n$, apply Lemma~\ref{lem:sqrt-separator} to find two separation components $G_1$ and $G_2$ of $G$, with separator $S$. Each separation component will have at least $n/3\ge (X/\epsilon)^2$ vertices, so we may apply the induction hypothesis to them, separating them into components with at most $\rhof/\epsilon^2$ vertices per component. Our separation of $G$ is the union of the sets of components in these two separations.

The separator for this separation consists of the union of the separators in the two subgraphs $G_1$ and $G_2$, together with the vertices in $S$. Let $n_1$ and $n_2$ be the numbers of vertices in $G_1$ and $G_2$, respectively. When we combine the two separations, the separator multiplicity of vertices outside $S$ remains unchanged, and the multiplicity of vertices in $S$ increases by exactly one.
Therefore, by the induction hypothesis, the combined total multiplicity is at most
$$(\epsilon n_1 - X\sqrt{n_1}) + (\epsilon n_2 - X\sqrt{n_2}) + \sigmaf\sqrt{n}.$$
The worst case, for the terms $X\sqrt{n_1}$ and $X\sqrt{n_2}$ appearing in this expression, is that $n_1$ is $n/3$ and $n_2$ is only a small amount larger than $2n/3$; for, in that case, the sum of these two terms is as small as possible. If $n_1$ and $n_2$ are closer together, the sum of these terms is larger and a larger amount is subtracted from the total expression. Substituting $n_1$ and $n_2$ for this worst case and using the assumption that $\epsilon\le 1$ to bound $\epsilon(n_1+n_2)\le \epsilon n + \sigmaf\sqrt n$ allows this upper bound on the total multiplicity of the separator to be regrouped as
$$\epsilon n - (X\sqrt{2/3} + X\sqrt{1/3} - 2\sigmaf)\sqrt{n}.$$
Due to our choice of $X$, this further simplifies to $\epsilon n - X\sqrt n$ as required by the induction.
\end{proof}

\section{Density-minimal graphs from limiting density}

As we now show, the limiting density of a minor-closed graph family $\Family$ may be approximated by the densities of a subfamily of the density-minimal graphs that it contains. Some care is needed, though, as it is possible for a minor-closed graph family to contain some density-minimal graphs whose densities are strictly larger than the limiting density of the family. The main idea of the sequence of lemmas in this section is to show that $\Family$ contains large dense graphs with a structure that is progressively more uniform, eventually leading to the result that it contains large dense density-minimal fans.

\begin{lemma}
\label{lem:dense-separation}
Let $\Family$ be a minor-closed family with limiting density $\Delta$. Then for any $\epsilon>0$ and any integer $k>0$, $\Family$ contains a graph $G$, such that $G$ has a separation into $k$ or more separation components, each of which has size $O(\epsilon^{-2})$ (where the constant hidden in the $O$-notation depends only on $\Family$), such that the disjoint union of the separation components forms a graph with density at least $\Delta-\epsilon$.
\end{lemma}

\begin{proof}
Let $\delta=\epsilon/(2\Delta)$.
Because $\Family$ has limiting density $\Delta$, it contains arbitrarily large graphs $G$ of density at least $\Delta-\epsilon/2$. By Lemma~\ref{lem:pulverize}, any such graph $G$ with $n$ vertices has a separation into separation components of size $O(\delta^{-2})=O(\epsilon^{-2})$, in which the total multiplicity of the separation is at most $\delta n$. For sufficiently large $n$, the number of components is at least $k$.
The density of the disjoint union of the separation components can be calculated by replacing the denominator, $n$, in the definition of the density of $G$ by the new denominator $n+\delta n$; therefore, the density of the disjoint union is at most $(\Delta-\epsilon/2)/(1+\delta)\ge\Delta-\epsilon$.
\end{proof}

A simple argument related to Chebyshev's inequality lets us replace the density of the disjoint union of the separation components (a weighted average of their individual densities) by a lower bound on the density of each separation component.

\begin{lemma}
\label{lem:component-density}
Let $\Family$ be a minor-closed family with limiting density $\Delta$. Then for any $\epsilon>0$ and any integer $k>0$, $\Family$ contains a graph $G$, such that $G$ has a separation into $k$ separation components of size $O(\epsilon^{-2})$ in which each component has density at least $\Delta-\epsilon$. As in Lemma~\ref{lem:dense-separation}, the hidden constant in the $O$-notation depends only on $\Family$.
\end{lemma}

\begin{proof}
By Lemma~\ref{lem:dense-separation}, for any $k'$ there is graph $G$ in $\Family$ with a separation into $k'$ separation components of size $O(\epsilon^{-2})$, such that the density of the disjoint union of the components is at least $\Delta-\epsilon/2$. But the density of the disjoint union is a weighted average of the densities of its separation components, weighted by the number of vertices in each component.  The weights of any two components differ by at most a factor of $O(\epsilon^{-2})$. Therefore, if $\Delta^*$ denotes the largest density of any graph of size $O(\epsilon^{-2})$ in $\Family$ (necessarily bounded since there are only finitely many graphs of that size) then, in order to achieve a weighted average of $\Delta-\epsilon/2$, every $O(2\epsilon^{-3}(\Delta^*-\Delta+\epsilon/2))$ separation components with density less than $\Delta-\epsilon$ (at least $\epsilon/2$ units below the average) must be balanced by at least one component with density greater than that threshold (and at most $\Delta^*-\Delta+\epsilon/2$ units above the average). By setting $k'$ sufficiently large we may find a a graph $G'$ in which at least $k$ of the separation components have density at least $\Delta-\epsilon$, and by deleting the lower-density components of $G'$ we can find a graph with $k$ dense separation components. As $G'$ is a subgraph of $G$, it must belong to~$\Family$.
\end{proof}

We define two separation components to be \emph{isomorphic} if they are isomorphic as labeled graphs, with a labeling of each vertex according to whether it is an internal vertex of the component or a separator vertex.

\begin{lemma}
\label{lem:isomorphic}
Let $\Family$ be a minor-closed family with limiting density $\Delta$. Then for any $\epsilon>0$ and any integer $k>0$, $\Family$ contains a graph $G$, such that $G$ has a separation into $k$ isomorphic separation components of size $O(\epsilon^{-2})$ in which each component's density is at least $\Delta-\epsilon$. The hidden constant in the $O$-notation depends only on $\Family$.
\end{lemma}

\begin{proof}
Let $s=O(\epsilon^{-2})$ be the bound on the component size for $\Family$ and $\epsilon$ given by Lemma~\ref{lem:component-density}, and let $N$ be the number of isomorphism classes of $s$-vertex labeled graphs in $\Family$ in which each vertex is labeled as an internal vertex or a separator vertex; note that, by Lemma~\ref{lem:component-density}, $N>0$. By Lemma~\ref{lem:component-density}, there is a graph $G'$ in $\Family$ with a separation into $kN$ components of size at most $s$, all of which have density at least $\Delta-\epsilon$. At least $k$ of these components must be isomorphic to each other; let $G$ be the union of those isomorphic components, with all the other components removed. Then $G$ has the stated properties, and as a subgraph of $G'$ it belongs to $\Family$.
\end{proof}

The labeling used in Lemma~\ref{lem:isomorphic} is not refined enough for our purposes. What we need is that, if $v$ is a vertex of the separator of $G$, then every separation component uses $v$ in the same way. More formally, we define a separation of a graph $G$ to be \emph{uniform} if every separation component in it is isomorphic to the same $s$-vertex graph (for some $s$), and if the vertices of $G$ can be labeled with the integers from $1$ to $s$ in such a way that all components are isomorphic as labeled graphs. Equivalently, a separation is uniform if (for some $t\le s$) the separation vertices may be given labels from $1$ to $t$ in such a way that the separation vertices within each separation component have distinct labels, and all pairs of separation components have an isomorphism that respects the labels. To achieve this, we use \emph{color coding}, a variant of the probabilistic method developed for solving subgraph isomorphism problems~\cite{AloYusZwi-JACM-95}.

\begin{lemma}
\label{lem:uniform}
Let $\Family$ be a minor-closed family with limiting density $\Delta$. Then for any $\epsilon>0$ and any integer $k>0$, $\Family$ contains a graph $G$, such that $G$ has a uniform separation into $k$ separation components of size $O(\epsilon^{-2})$ each of which has density at least $\Delta-\epsilon$. The hidden constant in the $O$-notation depends only on $\Family$.
\end{lemma}

\begin{proof}
Let $s=O(\epsilon^{-2})$ be the size of the components produced by Lemmas \ref{lem:component-density} and~\ref{lem:isomorphic}. By Lemma~\ref{lem:isomorphic}, there exists a graph $G'$ with a separation into $k(s-1)^{s-1}$ isomorphic separation components of size at most $s$ in which each component has density at least $\Delta-\epsilon$. Suppose that, for each separation component of $G'$, $t<s$ of the vertices are separator vertices, and the remaining $s-t$ vertices are internal to the component. We then choose, independently and uniformly at random for each vertex of the separator of $G'$, a label from $1$ to~$t$, and we also choose (arbitrarily rather than randomly) a labeling $L$ of a single isomorphic copy of the separation component (separate from $G'$) that places distinct labels from $1$ to $t$ on its separator vertices. For each separation component of $G'$, the probability is at least $t^{-t}$ that it has an isomorphism to $L$ such that the labels of the component in $G'$ match the labels in $L$, so the expected number of separation components with label-preserving isomorphisms of this type is at least $k(s-1)^{s-1}t^{-t}\ge k$. Therefore, there exists a labeling of $G'$ that matches or exceeds this expected value, and allows at least $k$ of the separation components of $G'$ to be given matching labels on their separator vertices. By keeping these $k$ matching components and deleting the rest, we find a subgraph $G$ of $G'$ that belongs to $\Family$ and has a uniform separation as required.
\end{proof}

Observe that, in a uniform separation of a simple graph, there can be no edges in which both endpoints belong to the separator, because that would result in the graph being a multigraph. Therefore, if the graph is connected, each separation component must have at least one non-separator vertex.
In the labeling defining a uniform separation of a graph $G$, we say that the label of a separator vertex is \emph{singular} if there is only one vertex in $G$ with that label, and \emph{plural} otherwise.

\begin{lemma}
\label{lem:singular}
Let $\Family$ be a minor-closed family with limiting density $\Delta$. Then for any $\epsilon>0$ and any integer $k>0$, $\Family$ contains a graph $G$, such that $G$ has a uniform separation into $k$ separation components of size $O(\epsilon^{-2})$ in which each component has density at least $\Delta-\epsilon$ and in which each separator vertex label is singular. The hidden constant in the $O$-notation depends only on $\Family$.
\end{lemma}

\begin{proof}
Let $s=O(\epsilon^{-2})$ be the maximum size of the separation components produced by Lemma~\ref{lem:uniform} for $\Family$ and $\epsilon$. By Lemma~\ref{lem:uniform}, there exists a graph $G'$ in $\Family$ with a uniform separation into $k^{2^s}$ separation components of size at most $s$ in which each component has density at least $\Delta-\epsilon$.
We then consider each label of a separator vertex in $G'$ in turn; for each such label, we either make it singular or make it not be a separator vertex any more, at the expense of reducing the number of separation components to the square root of its former value. Our choice of $k^{2^s}$ separation components at the start of this process ensures that there will be at least $k$ separation components when we have completed the process.

So, suppose that we have $p^2$ separation components in which label $\lambda$ is plural, and we wish to find some subset of $p$ of the components such that, for the graph induced by that subset, $\lambda$ is either singular or not a separator vertex. We consider the number of different separator vertices that have the label $\lambda$. If this number is at least $p$, then we may choose a single separation component for each different vertex with that label, ending up with at least $p$ separation components overall;  the graph induced by these separation components remains uniformly separated and, in its separation, the vertices labeled $\lambda$ are no longer separation vertices because each belongs to a single separation component. Making these vertices into internal vertices of the separation components does not change the density of the components.

In the other case, there are fewer than $p$ separator vertices with label $\lambda$. Therefore, because there are $p^2$ separation components, one of the vertices labeled $\lambda$ must be shared by at least $p$ separation components. The graph induced by these separation components remains uniformly separated and in it $\lambda$ is singular.

Repeating this refinement process once for each of the labels of separator vertices produces a uniform separation in which all separator vertex labels are singular, as required.
\end{proof}

A graph with a uniform separation in which each separator vertex label is singular is just a more complicated way of describing a fan, and the density of the fan is at least as large as the density of the separation components it is formed from. However, in order to apply Lemma~\ref{lem:fan-minimality} we need the fan to have some more structure: the shared vertices of the fan should form a clique, the repeated parts of the fan should be connected, and each shared vertex should be adjacent to a repeated vertex.

\begin{lemma}
\label{lem:dense-fan}
Let $\Family$ be a minor-closed family with limiting density $\Delta$. Then for any $\epsilon>0$ there exists a graph $G$, and a proper subset $S$ of the vertices of $G$ such that
$\Fan(G,S,k)\in\Family$ for all positive $k$, such that $G\setminus S$ is connected and $S$ forms a clique, such that each vertex of $S$ is adjacent to a vertex in $G$, and such that for all sufficiently large $k$ the density of $\Fan(G,S,k)$ is at least $\Delta-\epsilon$.
\end{lemma}

\begin{proof}
By Lemma~\ref{lem:singular} we can find, for any $k$, a graph $G$ with $O(\epsilon^{-2})$ vertices and a proper subset $S$ such that $G$ has density at least $\Delta-\epsilon/2$ and $\Fan(G,S,k)$ belongs to $\Family$. By choosing $G$ appropriately, we can additionally guarantee that $\Fan(G,S,k')$ also belongs to $\Family$ for all $k'>k$; for, if each of the finitely many choices for $G$ had a fixed bound on the number of times it could be repeated in a fan, this would contradict Lemma~\ref{lem:singular} for values of $k$ that exceeded the maximum of these bounds. Because $G$ and $\Fan(G,S,k)$ come from a uniform separation, there are no edges in the subgraph induced by $S$. Let $m$ be the number of edges in $G$, $n$ be the number of vertices in $G$, $s$ be the number of vertices in $S$, and $m_i$ and $n_i$ be the number of edges incident to and vertices in each connected component $G_i$ of $G\setminus S$. Then $m/(n-s)\ge m/n\ge\Delta-\epsilon/2$ and $m/(n-s)$ is a weighted average of the quantities $m_i/n_i$, so if $G_i$ is the component of $G\setminus S$ maximizing $m_i/n_i$ then $m_i/n_i\ge\Delta-\epsilon/2$.

Let $S_i$ be the set of vertices in $S$ incident to $G_i$, and let $G'$ have $G_i\cup S_i$ as its vertex set and have as its edges all the edges in $G$ that are incident to $G_i$ together with an edge between every two vertices in $S_i$. By Lemma~\ref{lem:fan-clique}, $\Fan(G',S_i,k)$ belongs to $\Family$ for all $k$. As required, $S_i$ forms a clique, $G'\setminus S_i$ is connected, and every vertex in $S_i$ has a neighbor that is not in $S_i$. As $k$ grows larger, the density of $\Fan(G',S_i,k)$ converges to $m_i/n_i\ge\Delta-\epsilon/2$, so for all sufficiently large $k$ it is larger than $\Delta-\epsilon$.
\end{proof}

Dense fans in $\Family$ can be used to generate density-minimal graphs whose densities approximate the limiting density of $\Family$.

\begin{lemma}
\label{lem:density}
Let $\Family$ be a minor-closed family with limiting density $\Delta$. Then for any $\epsilon>0$ there exists a density-minimal graph $G$ in $\Family$ whose density belongs to the closed interval $[\Delta-\epsilon,\Delta]$.
\end{lemma}

\begin{proof}
By Lemma~\ref{lem:dense-fan} we can find $G$, $S$, and $k_0$ such that
$\Fan(G,S,k)\in\Family$ for all positive $k$, such that $\Fan(G,S,k)$ meets the conditions of Lemma~\ref{lem:fan-minimality}, and such that the density of $\Fan(G,S,k)$ is at least $\Delta-\epsilon$ for $k\ge k_0$.
By Lemma~\ref{lem:jump} there exists $\delta$ such that there are no density-minimal graphs with densities in the open interval $(\Delta,\Delta+\delta)$.
$\Family$ may contain graphs with density $\Delta+\delta$ or larger, but only finitely many of them; let $k_1$ be the number of vertices in the largest such graph, and set $k=\max(k_0,k_1+1)$.
By Lemma~\ref{lem:fan-minimality} there exist $G'$ and $S'$ such that $\Fan(G',S',k)$ is a densest minor of $\Fan(G,S,k)$, and therefore also has density at least $\Delta-\epsilon$; since $\Fan(G',S',k)$ has at least $k\ge k_1+1$ vertices (Lemma~\ref{lem:fan-lb}), it has density less than $\Delta+\delta$ and therefore its density is at most $\Delta$.
According to Lemma~\ref{lem:fan-minimality}, $G'$ and $S'$ can be chosen so that either $\Fan(G',S',k)$ is density-minimal, or $G'$ is density-minimal and has the same density as the fan; in either case we have found a density-minimal graph in $\Family$ whose density lies in the desired range.
\end{proof}

\section{Main results}

\begin{theorem}
\label{thm:cwoc}
The set of limiting densities of minor-closed graph families is countable, well-ordered, and topologically closed.
\end{theorem}

\begin{proof}
By the Robertson--Seymour theorem~\cite{RobSey-JCTSB-04}, according to which every minor-closed graph family can be characterized by a finite set of forbidden minors, there are countable many minor-closed families and therefore at most countably many limiting densities.

There can be no infinite descending sequence of limiting densities of minor-closed families, because if there were, there would be an infinite descending subsequence $\Delta_1$, $\Delta_2$,  $\Delta_3$, that converged to some limit $\Delta^*$. But then, for every $\delta$, there would be a limiting density $\Delta_i$ within $\delta/2$ of $\Delta^*$, and (by Lemma~\ref{lem:density}) a density-minimal graph $G$ with density within $\delta/2$ of $\Delta_i$, contradicting Lemma~\ref{lem:jump} according to which there is an open interval $(\Delta,\Delta+\delta)$ that does not contain the density of any density-minimal graph. Therefore, the set of limiting densities is well-ordered.

By well-ordering, any cluster point $\Delta$ of the set of limiting densities must be the limit of an increasing subsequence $\Delta_1<\Delta_2<\Delta_3\dots$ of limiting densities of minor-closed families $\Family_1$, $\Family_2$, $\Family_3$,~$\dots$. But then $\Family_1\cup\Family_2\cup\Family_3\dots$ is also minor-closed and has $\Delta$ as its limiting density; hence the set of limiting densities contains all its cluster points and is topologically closed.
\end{proof}

\begin{theorem}
\label{thm:dd}
The set of limiting densities of minor-closed graph families is the topological closure of the set of densities of density-minimal graphs.
\end{theorem}

\begin{proof}
By Corollary~\ref{cor:cf}, the set of limiting densities of minor-closed graph families contains the set of densities of density-minimal graphs, and by Theorem~\ref{thm:cwoc} it contains the closure of this set.
By Lemma~\ref{lem:density} every limiting density of a minor-closed graph family belongs to the closure of the set of densities of density-minimal graphs.
\end{proof}

\begin{theorem}
\label{thm:increment}
Let $\Delta$ be the limiting density of a minor-closed graph family $\Family$. Then  $1+\Delta$ is a cluster point in the set of limiting densities.
\end{theorem}

\begin{proof}
Let $\Sparse(\Delta)$ be defined as in the proof of Lemma~\ref{lem:jump} as the family of graphs with no minor denser than $\Delta$. Define $\Apex(\Delta,k)$ to be the family of graphs $G$ with the following two properties:
\begin{itemize}
\item Each connected component $G_i$ of $G$ has at most $k$ vertices, and
\item Within each connected component $G_i$ of $G$, one can find a vertex $a_i$ (the \emph{apex} of the component) such that $(G_i\setminus\{a_i\})\in \Sparse(\Delta)$.
\end{itemize}
$\Apex(\Delta,k)$ is minor-closed, and as we show below, the sequence of graph families $\Apex(\Delta,k)$ (for fixed $\Delta$ and variable $k$) has a sequence of limiting densities that converges to $1+\Delta$ but that does not ever actually reach $1+\Delta$.

Because the density of any graph in $\Sparse(\Delta)$ is at most $\Delta$, any connected component of a graph in $\Apex(\Delta,k)$ that has $i$ vertices has at most $\Delta(i-1)$ edges that are not incident to $a_i$, and another $i-1$ edges incident to $a_i$. Therefore, its density is at most $(1+\Delta)\frac{i-1}{i}$, and the limiting density of $\Apex(\Delta,k)$ is at most $(1+\Delta)\frac{k-1}{k}$. In particular, it cannot be the case that any family $\Apex(\Delta,k)$ has limiting density exactly equal to $\Delta+1$.

For any $\epsilon$ there exists a density-minimal graph $G$ in $\Family$ with density in the interval $[\Delta-\epsilon/2,\Delta]$, by Lemma~\ref{lem:density}. Because $G$ is density-minimal and has density at most $\Delta$, it belongs to $\Sparse(\Delta)$. Let $G$ have $m$ edges and $n$ vertices, and form a graph $G'$ by connecting each vertex in $G$ to a new vertex $a$.
The graphs $\Fan(G',\{a\},k)$ have density $(m+n)k/(nk+1)$, which approaches $1+m/n$ as $k$ becomes large, so for some sufficiently large $k$ the graph $H=\Fan(G',\{a\},k)$ has density at least $1+\Delta-\epsilon$. The graphs $\Fan(H,\emptyset,r)$ belong to $\Apex(\Delta,nk+1)$: each component of these graphs has $nk+1$ vertices, meeting the limit on the component size in $\Apex(\Delta,nk+1)$, and within each component of $\Fan(H,\emptyset,r)$ the vertex $a$ can be chosen as the apex. The family $\Apex(\Delta,nk+1)$ contains arbitrarily large graphs
$\Fan(H,\emptyset,r)$ with density at least $1+\Delta-\epsilon$, so its limiting density is within the interval $[1+\Delta-\epsilon,(1+\Delta)\frac{nk}{nk+1}]$ and the sequence of limiting densities of these families approaches $1+\Delta$.

We have found a sequence of minor-closed families of graphs with limiting densities approaching but not equalling $1+\Delta$, so $1+\Delta$ is a cluster point in the set of limiting densities.
\end{proof}

We say that a number $\Delta$ is an order-1 cluster point in the set of limiting densities if $\Delta$ is any cluster point of the set, and that $\Delta$ is an order-$i$ cluster point if $\Delta$ is a cluster point of order-$(i-1)$ cluster points.

\begin{theorem}
\label{thm:bigset}
The subset of limiting densities of minor-closed graph families that are less than $3/2$ is the set
$$\BigSet .$$
For the integers $i\ge 1$ and $j\ge 1$, the numbers $i+(j-1)/j$ are order-$i$ cluster points  in the set of limiting densities.
\end{theorem}

\begin{proof}
The description of the limiting densities below $3/2$ follows from Theorem~\ref{thm:dd} and Lemma~\ref{lem:low-density}. The characterizations of the numbers $i+(j-1)/j$ as order-$i$ cluster points follows by induction on~$i$ using Theorem~\ref{thm:increment}, with the description of the subset of limiting densities that are less than or equal to 1 as a base case for the induction.
\end{proof}

\section{Multigraphs}

Instead of using simple graphs,
the theory of graph may be formulated in terms of \emph{multigraphs}, graphs in which two vertices may be connected by a \emph{bond} of two or more edges and in which a single vertex may have any number of \emph{self-loops}, edges that have only that vertex as its endpoints. In some ways, the theory for multigraphs is simpler: for instance, the forbidden minor for forests in the simple graph world is a triangle, whereas for forests it is a smaller graph with one vertex and one self-loop. This simplicity carries over to the theory of limiting densities as well.

\begin{theorem}
\label{thm:multi}
The limiting density of a minor-closed family $\Family$ of multigraphs can only be an integer, a superparticular ratio $i/(i+1)$ for a positive integer $i$, or unbounded.
\end{theorem}

\begin{proof}
If $\Family$ includes bonds with arbitrarily large numbers of edges, then its limiting density is unbounded. Otherwise, there can be at most finitely many different multigraphs in $\Family$ that have a fixed number of vertices, the only ingredient needed to make the counting arguments in our proofs go through. So, multigraphs in $\Family$ obey a separator theorem (obtained by applying the Alon--Seymour--Thomas separator theorem to the underlying simple graph of the multigraph), they contain dense fans of multigraphs, and their densities are approximated by the densities of finite density-minimal multigraphs. But if a connected multigraph contains a cycle, self-loop, or bond, then its densest minor is just a single vertex with density equal to its rank, because in the multigraph world edge contractions do not lead to the removal of any other edges, and therefore can only increase the density of a graph with at least as many edges as vertices. Therefore, the only density-minimal multigraphs are the trees and the single-vertex multigraphs, and the only possible densities that can be obtained from them are the ones in the statement of the theorem.
\end{proof}

More specifically, if $\Family$ can have a single connected component with arbitrarily high rank, its limiting density is unbounded. Otherwise, let $r$ be the largest number such that graphs in $\Family$ can have an unbounded number of rank-$r$ connected components. If $r$ is nonzero, the limiting density of $\Family$ is $r$. If $r$ is zero, and $\Family$ can have a single connected component with unbounded size, its limiting density is one. And if $r$ is zero and all components of $\Family$ have bounded size, the limiting density of $\Family$ is $i/(i+1)$, where $i$ is the number of edges in the largest tree $T$ such that $\Fan(T,\emptyset,k)\in\Family$ for all $k$.

\section{Conclusions}

We have investigated the structure of the set of limiting densities of minor-closed graph families; this set is topologically closed and well-ordered, contains cluster points of cluster points, and its exact members are known up to the threshold $3/2$. However, our results open many additional questions for investigation:
\begin{itemize}
\item Are all limiting densities rational?
\item Is every limiting density equal to the density of a density-minimal graph, or are there cluster points in the set of limiting densities that are not themselves densities of density-minimal graphs? A positive answer to this question would also imply that all limiting densities are rational.
\item If $x=p/q$ is a limiting density, where $p$ and $q$ are integers,  can the size of the gap between $x$ and the next larger limiting density be lower-bounded by a nonzero closed-form expression in terms of $p$ and $q$? Due to the existence of cluster points, some dependence on $q$ seems to be necessary: a monotonic function of $x$ alone would have to approach zero as $x$ approaches the first cluster point, and could not give a nontrivial lower bound on the gaps for higher values of $x$.
\item Are there any other cluster points between $3/2$ and $2$ other than the ones of the form $2-1/i$ that we have already identified?
\item What is the order type of the set of limiting densities? Since the limiting densities contain cluster points of any finite order, the order type is at least $\omega^\omega$; is this the exact order type?
\item What is the smallest cluster point of order $i$ for each integer~$i$? Is it the number $i$ itself?
\item If $\Delta$ is a cluster point of limiting densities, must $\Delta-1$ be a limiting density? If so, it would follow that all limiting densities are rational, that $i$ is the smallest cluster point of order~$i$, and that the order type of the set of limiting densities is $\omega^\omega$.
\end{itemize}

\subsection*{Acknowledgements}

This work was supported in part by NSF grant
0830403 and by the Office of Naval Research under grant
N00014-08-1-1015.  We thank {\'E}ric Fusy for calling our attention to reference~\cite{BerNoyWel-JCTB-10}, and an anonymous referee for many helpful suggestions.

\raggedright
\bibliographystyle{abuser}
\bibliography{minor-densities}

\begin{thebibliography}{10}

\bibitem{AloSeyTho-JAMS-90}
N.~Alon, P.~Seymour, and R.~Thomas.
\newblock {A separator theorem for nonplanar graphs}.
\newblock {\em J. Amer. Math. Soc.} 3(4):801{--}808, 1990,
  \href{http://dx.doi.org/10.2307/1990903}%
{doi:10.2307/1990903}.

\bibitem{AloYusZwi-JACM-95}
N.~Alon, R.~Yuster, and U.~Zwick.
\newblock {Color-coding}.
\newblock {\em J. ACM} 42(4):844{--}856, 1995,
  \href{http://dx.doi.org/10.1145/210332.210337}%
{doi:10.1145/210332.210337}.

\bibitem{BerNoyWel-JCTB-10}
O.~Bernardi, M.~Noy, and D.~Welsh.
\newblock {Growth constants of minor-closed classes of graphs}.
\newblock {\em J. Combinatorial Theory, Ser. B} 100(5):468{--}484, 2010,
  \href{http://dx.doi.org/10.1016/j.jctb.2010.03.001}%
{doi:10.1016/j.jctb.2010.03.001}.

\bibitem{ChuReeSey-ms-08}
M.~Chudnovsky, B.~Reed, and P.~Seymour.
\newblock {The edge-density for $K_{2,t}$ minors},
  \url{http://www.math.princeton.edu/~pds/papers/bipminors/paper.pdf}.
\newblock 2008.

\bibitem{Die-05}
R.~Diestel.
\newblock {\em {Graph Theory}}.
\newblock Graduate Texts in Mathematics 173. Springer-Verlag, 2005,
  \url{http://diestel-graph-theory.com/}.

\bibitem{ErdRenSos-SSMH-66}
P.~Erd{\H{o}}s, A.~R{\'e}nyi, and V.~T. S{\'o}s.
\newblock {On a problem of graph theory}.
\newblock {\em Studia Sci. Math. Hungar.} 1:215{--}235, 1966,
  \url{http://www.renyi.hu/~p_erdos/1966-06.pdf}.

\bibitem{Kos-Comb-84}
A.~V. Kostochka.
\newblock {Lower bound of the Hadwiger number of graphs by their average
  degree}.
\newblock {\em Combinatorica} 4:307{--}316, 1984,
  \href{http://dx.doi.org/10.1007/BF02579141}%
{doi:10.1007/BF02579141}.

\bibitem{KosPri-DM-08}
A.~V. Kostochka and N.~Prince.
\newblock {On $K_{s, t}$-minors in graphs with given average degree}.
\newblock {\em Discrete Math.} 308(19):4435{--}4445, 2008,
  \href{http://dx.doi.org/10.1016/j.disc.2007.08.041}%
{doi:10.1016/j.disc.2007.08.041}.

\bibitem{KosPri-DM-10}
A.~V. Kostochka and N.~Prince.
\newblock {Dense graphs have $K_{3,t}$ minors}.
\newblock {\em Discrete Math.}, 2010,
  \href{http://dx.doi.org/10.1016/j.disc.2010.03.026}%
{doi:10.1016/j.disc.2010.03.026}.

\bibitem{LipTar-SJC-80}
R.~J. Lipton and R.~E. Tarjan.
\newblock {Applications of a planar separator theorem}.
\newblock {\em SIAM J. Comput.} 9(3):615{--}627, 1980,
  \href{http://dx.doi.org/10.1137/0209046}%
{doi:10.1137/0209046}.

\bibitem{Mad-MA-67}
W.~Mader.
\newblock {Homomorphieeigenschaften und mittlere Kantendichte von Graphen}.
\newblock {\em Math. Ann.} 174:265{--}268, 1967.

\bibitem{MyeTho-Comb-05}
J.~S. Myers and A.~Thomason.
\newblock {The extremal function for noncomplete minors}.
\newblock {\em Combinatorica} 25(6):725{--}753, 2005,
  \href{http://dx.doi.org/10.1007/s00493-005-0044-0}%
{doi:10.1007/s00493-005-0044-0}.

\bibitem{RobSey-JCTSB-04}
N.~Robertson and P.~D. Seymour.
\newblock {Graph Minors. XX. Wagner's conjecture}.
\newblock {\em J. Combinatorial Theory, Ser. B} 92(2):325{--}357, 2004,
  \href{http://dx.doi.org/10.1016/j.jctb.2004.08.001}%
{doi:10.1016/j.jctb.2004.08.001}.

\bibitem{SchZit-JCtB-94}
E.~Scheinerman and J.~Zito.
\newblock {On the size of hereditary classes of graphs}.
\newblock {\em J. Combinatorial Theory, Ser. B} 61(1):16{--}39, 1994,
  \href{http://dx.doi.org/10.1006/jctb.1994.1027}%
{doi:10.1006/jctb.1994.1027}.

\bibitem{SonTho-JCTB-06}
Z.-X. Song and R.~Thomas.
\newblock {The extremal function for $K_9$ minors}.
\newblock {\em J. Combinatorial Theory, Ser. B} 96(2):240{--}252, 2006,
  \href{http://dx.doi.org/10.1016/j.jctb.2005.07.008}%
{doi:10.1016/j.jctb.2005.07.008}.

\bibitem{TazMul-DAM-09}
S.~Tazari and M.~M{\"u}ller-Hannemann.
\newblock {Shortest paths in linear time on minor-closed graph classes, with an
  application to Steiner tree approximation}.
\newblock {\em Discrete Applied Math.} 157(4):673{--}684, 2009,
  \href{http://dx.doi.org/10.1016/j.dam.2008.08.002}%
{doi:10.1016/j.dam.2008.08.002}.

\bibitem{ThoWol-JCTB-08}
R.~Thomas and P.~Wollan.
\newblock {The extremal function for 3-linked graphs}.
\newblock {\em J. Combinatorial Theory, Ser. B} 98(5):939{--}971, 2008,
  \href{http://dx.doi.org/10.1016/j.jctb.2007.11.008}%
{doi:10.1016/j.jctb.2007.11.008}.

\bibitem{Tho-MPCPS-84}
A.~Thomason.
\newblock {An extremal function for contractions of graphs}.
\newblock {\em Math. Proc. Camb. Phil. Soc.} 95(2):261{--}265, 1984,
  \href{http://dx.doi.org/10.1017/S0305004100061521}%
{doi:10.1017/S0305004100061521}.

\bibitem{Tho-JCTB-01}
A.~Thomason.
\newblock {The extremal function for complete minors}.
\newblock {\em J. Combinatorial Theory, Ser. B} 81(2):318{--}338, 2001,
  \href{http://dx.doi.org/10.1006/jctb.2000.2013}%
{doi:10.1006/jctb.2000.2013}.

\bibitem{Whi-TAMS-32}
H.~Whitney.
\newblock {Non-separable and planar graphs}.
\newblock {\em Trans. Amer. Math. Soc.} 34(2):339{--}362, 1932,
  \href{http://dx.doi.org/10.2307/1989545}%
{doi:10.2307/1989545}.

\end{thebibliography}
\end{document}